\documentclass[11pt]{amsart}

\usepackage{latexsym,rawfonts}
\usepackage{amsfonts,amssymb}
\usepackage{amsmath,amsthm}

\usepackage[plainpages=false]{hyperref}
\usepackage{graphicx}
\usepackage{color}
\usepackage[table]{xcolor}
\usepackage{longtable}

\newcommand{\rn}{\mathbb R}

\newcommand{\rnnn}{\mathbb R^{n+1}}

\newcommand{\rt}{\mathbb R_{+}}
\newcommand{\sn}{ {\mathbb{S}^{n}}}

\newcommand{\psum}{{+_{\negthinspace\kern-2pt p}}\,}
\newcommand{\qsum}[1]{{+_{\negthinspace\kern-2pt #1}}\,}
\newcommand{\dpsum}{{\tilde+_{\negthinspace\kern-1pt p}}\,}
\newcommand{\dqsum}[1]{{\tilde+_{\negthinspace\kern-1pt #1}}\,}
\newcommand{\lsub}[1]{\hskip -1.5pt\lower.5ex\hbox{$_{#1}$}}

\numberwithin{equation}{section}

\newtheorem{theo}{Theorem}[section]
\newtheorem{coro}[theo]{Corollary}
\newtheorem{lem}[theo]{Lemma}

\newtheorem{rem}[theo]{Remark} \theoremstyle{definition}

\begin{document}

\title{Uniqueness of solutions to the isotropic $L_{p}$ Gaussian Minkowski problem
}
\author[J. Hu]{Jinrong Hu}
\address{Institut f\"{u}r Diskrete Mathematik und Geometrie, Technische Universit\"{a}t Wien, Wiedner Hauptstrasse 8-10, 1040 Wien, Austria
 }
\email{jinrong.hu@tuwien.ac.at}
\begin{abstract}
The uniqueness of solutions to the isotropic $L_{p}$ Gaussian Minkowski problem in $\rnnn$ is established when $-(n+1)<p<-1$ with $n\geq 1$, without requiring the origin-centred assumption on convex bodies.
\end{abstract}


\thanks{This work was supported by the Austrian Science Fund (FWF): Project P36545.}

\maketitle

\baselineskip18pt

\parskip3pt

\section{Introduction}

The Minkowski problem prescribing the surface area measure stands as one of the fundamental problems in the Brunn-Minkowski theory, which was originally proposed and addressed by Minkowski himself \cite{M897,M903}. This problem has seen significant progress through a series of papers (see, e.g. ~\cite{A39,A42,B87,FJ38}). As an important extension of the Minkowski problem, the $L_p$ Minkowski problem characterizing the $L_{p}$ surface area measure was introduced by Lutwak ~\cite{L93} and lies at the heart of the $L_{p}$ Brunn-Minkowski theory. Building upon Lutwak's seminal work, the $L_{p}$ Minkowski problem has been the fertile ground that has yielded many meaningful results (see, e.g. ~\cite{B19,BLYZ12,B17,CL17,CL19,L04,Zhu14,Zhu15,Zh15}).

Analogues of the aforementioned Minkowski type problems and $L_{p}$ Minkowski type problems have emerged in probability theory. The basic geometric invariant in Gaussian probability space is the \emph{Gaussian volume}. Given a convex body $K$ in $\rnnn$, the Gaussian volume $\gamma(K)$ of $K$ is defined by
\[
\gamma(K)=\frac{1}{(\sqrt{2\pi})^{n+1}}\int_{K}e^{-\frac{|X|^{2}}{2}}dX.
\]
 Unlike Lebesgue measure, Gaussian volume is neither translation invariant nor homogeneous. Huang-Xi-Zhao \cite{HXYZ21} established the variational formula of the Gaussian volume and introduced the Gaussian surface area measure as
\begin{equation*}
S(K,\eta)=\frac{1}{(\sqrt{2\pi})^{n+1}}\int_{\nu^{-1}_{K}(\eta)}e^{-\frac{|X|^{2}}{2}}d\mathcal{H}^{n}(X)
	\end{equation*}
for each Borel subset $\eta\subset \sn$.  Here $\mathcal{H}^{n}$ is the $n$-dimensional Hausdorff measure, $\nu_{K}$ is the Gauss map defined on the subset of those points of $\partial K$, and $\nu^{-1}_{K}$ is the inverse Gauss map (see Sec. \ref{Sec2} for a detailed definition). Based on the above fact, they proposed the Gaussian Minkowski problem prescribing the Gaussian surface area measure and derived a normalized solution to the Gaussian Minkowski problem by using  variational arguments \cite{HLYZ10,HLYZ16} and moreover obtained the existence of weak solutions to the non-normalized Gaussian Minkowski problem with the aid of the degree theory method, under the condition that the Gaussian volume exceeds 1/2. Subsequently, research on the Gaussian Minkowski problem has been prolific. As a natural extension of the Gaussian Minkowski problem, the $L_{p}$ Gaussian Minkowski problem  has been  widely studied, such as see \cite{CHW23,FH23,HJ24,IM23,KL23,Lu22,Sh23}. This problem acts as a bridge which connects the $L_{p}$ Minkowski problem to Gaussian probability space, which prescribes the $L_{p}$ Gaussian surface area measure defined as
\begin{equation*}
	S_{p}(K,\eta)=\frac{1}{(\sqrt{2\pi})^{n+1}}\int_{\nu^{-1}_{K}(\eta)}(\langle X, \nu_{K}(X)\rangle)^{1-p}e^{-\frac{|X|^{2}}{2}}d\mathcal{H}^{n}(X)
	\end{equation*}
for each Borel subset $\eta\subset \sn$. If the $L_{p}$ Gaussian surface area measure is proportional to the  spherical Lebesgue measure, the solvability of the $L_{p}$ Gaussian Minkowski problem amounts to attacking the following {M}onge-{A}mp\`ere equation:
\begin{equation}\label{MP}
h^{1-p}\frac{1}{\kappa}e^{-\frac{|Dh|^{2}}{2}}=c, \ {\rm for} \ \  c>0, \quad {\rm on}\ \sn.
\end{equation}

The uniqueness  of solutions to \eqref{MP} plays a critical role in establishing the existence of solutions to the $L_{p}$ Gaussian Minkowski problem with the aid of degree theory, which is helpful to eliminate the Lagrange multiplier that appears in variational approaches. There are the known uniqueness results for the isotropic $L_{p}$ Gaussian Minkowski problem. When $n=1$, $p=1$ and $h\geq 0$, Chen-Hu-Liu-Zhao \cite{CHW23} proved the uniqueness of solutions to \eqref{MP}, building upon an argument of Andrews \cite{AN03}; as an application, this uniqueness result is used to prove the existence of smooth small solutions to the Gaussian Minkowski problem. Later, following similar lines as \cite{CHW23}, Liu \cite{LW24} proved the uniqueness by generalizing the results of \cite{CHW23} to the case $0\leq p<1$ and $n=1$. When $n\geq 2$, Ivaki \cite{Iva23} confirmed Chen-Hu-Liu-Zhao's conjecture by proving the uniqueness for $p=1$ under no extra assumption on the sign of $h$ and for $p>1$ but allowing $h>0$ by utilizing the Heintze-Karcher inequality. Moreover, for $p> -n-1$, uniqueness was shown in the class of origin-centred convex bodies via the local Brunn-Minkowski inequality derived by Ivaki-Milman \cite{IM23}. With an additional assumption on the size of the Gaussian volume (specially, $\geq$ 1/2), the uniqueness for $p\geq 1$ follows directly from the Ehrhard inequality and the characterization of its equality cases, for more details, such as see Shenfeld-Van Handel \cite{SV18} and \cite{FH23,HXYZ21,Lu22}. It is of great interest to investigate whether the origin-centred assumption could be removed for $p<1$ in higher dimensions.

Inspired by \cite{HI24,Iva23,IM23},  here we give the uniqueness result of solutions to \eqref{MP} in the case of $-(n+1)<p<-1$ without the origin-centred assumption on convex bodies by substituting a new suitable test function into the spectral formulation of the local Brunn-Minkowski inequality. Given a convex body $K$ in $\rnnn$, set $R(K)=\max|X|$ for $X\in  K$.  The main result is as follows.
 \begin{theo}\label{MTH}
Let $n\geq 1$. Suppose $-(n+1)< p <-1$. Let $\partial K$ be a smooth, strictly convex hypersurface with the support function $h>0$ and $R(K)\leq1$ such that $h^{1-p}e^{-\frac{|Dh|^{2}}{2}}\frac{1}{\kappa}=c$ for  $c>0$.  Then $\partial K$ is a sphere. In particular, if $c\in (0,e^{-1/2}]$,  there is a constant solution; if $c>e^{-1/2}$, there is no constant solution.

 \end{theo}

 It should be noted that, due to the limitation on the spectral form of the local Brunn-Minkowski inequality, this technique is not applicable for cases where $R(K)>1$. Indeed, the condition $R(K)\leq 1$ in Theorem \ref{MTH} arises naturally within this framework. To see this aspect, it is clear that for $-(n+1)<p<-1$, the density of the $L_{p}$ Gaussian surface area measure of a centered ball of radius $r$ is $\frac{1}{(\sqrt{2\pi})^{n+1}}e^{-\frac{r^{2}}{2}}r^{n+1-p}$. Observe that when $r\rightarrow 0^{+}$ or $r\rightarrow\infty$, $e^{-\frac{r^{2}}{2}}r^{n+1-p}\rightarrow 0$. This asymptotic behavior indicates that solutions to the related $L_{p}$ Gaussian Minkowski problem may not be unique without additional conditions.  Extending Theorem \ref{MTH} would require developing alternative tools, and we may provide some heuristic insights on this topic in Appendix A. Furthermore, the uniqueness result in Theorem \ref{MTH} is essential to ensure the well-definedness of degree theory (see, e.g \cite{LY89}) and to validate the feasibility of the degree-theoretic approach in Gaussian distribution settings, as first shown by Huang-Xi-Zhao \cite{HXYZ21}, thereby helps study the existence of small solutions to the non-normalized $L_{p}$ Gaussian Minkowski problem in the non-symmetric case for $-(n+1)<p<-1$.

\section{Preliminaries}
\label{Sec2}

The theory of convex bodies is well-covered in some standard references, such as the books of  Gardner \cite{G06} and Schneider \cite{S14}.

Write ${\rnnn}$ for the $(n+1)$-dimensional Euclidean space.   For $Y, Z\in {\rnnn}$, $ \langle Y,Z\rangle $ denotes the standard inner product. For vectors $X\in{\rnnn}$, $|X|=\sqrt{ \langle X, X}\rangle$ is the Euclidean norm.  Let ${\sn}$ be the unit sphere. A convex body is a compact convex set of ${\rnnn}$ with non-empty interior.

 Given a convex body $K$ in $ \rnnn$, for $x\in{\sn}$, the support function of $K$ (with respect to the origin) is defined by
\[
h(x):=h_{K}(x)=\max\{\langle x, Y\rangle:Y \in K\}.
\]

 Given a convex body $K$ in $\rnnn$, for $\mathcal{H}^{n}$ almost all $X\in \partial K$, the unit outer normal of $K$ at $X$ is unique. In such case, denote by $\nu_{K}$ the Gauss map that takes $X\in \partial K$ to its unique unit outer normal.
  For $\omega\subset {\sn}$, the inverse Gauss map $\nu_{K}$ is expressed as
\begin{equation*}
\nu^{-1}_{K}(\omega)=\{X\in \partial K:  \nu_{K}(X) {\rm \ is \ defined \ and }\ \nu_{K}(X)\in \omega\}.
\end{equation*}

For a smooth and strictly convex body $K$, i.e., its boundary is $C^{\infty}$-smooth and is of positive Gauss curvature, we abbreviate $\nu^{-1}_{K}$ as $F$ for brevity. Then the support function of $K$ can be written as
\begin{equation}\label{hhom}
h(x)=\langle x, F(x)\rangle=\langle\nu_{K}(X), X\rangle, \ {\rm where} \ x\in {\sn}, \ \nu_{K}(X)=x \ {\rm and} \ X\in \partial K.
\end{equation}
 Let $\{e_{1},e_{2},\ldots, e_{n}\}$ be a local orthonormal frame on ${\sn}$, $h_{i}$ and $h_{ij}$ be the first-order and second-order covariant derivatives of $h(\cdot)$  with respect to a local orthonormal frame on ${\sn}$. Differentiating \eqref{hhom} with respect to $e_{i}$, we get
\[
h_{i}=\langle e_{i}, F(x)\rangle+\langle x, F_{i}(x)\rangle.
\]
Since $F_{i}$ is tangent to $ \partial K$ at $F(x)$, we obtain
\begin{equation}\label{Fi}
h_{i}=\langle e_{i}, F(x)\rangle.
\end{equation}
By \eqref{hhom} and \eqref{Fi},
\begin{equation}\label{Fdef}
F(x)=\sum_{i} h_{i}e_{i}+hx=\nabla h+hx,
\end{equation}
where $\nabla$ is the (standard) spherical gradient. By extending $h(\cdot)$ to $\rnnn$ as a 1-homogeneous function and restricting the gradient of $h(\cdot)$ on $\sn$ to yield (see, e.g. \cite{CY76})
\begin{equation}\label{hf}
D h(x)=F(x), \ \forall x\in{\sn},
\end{equation}
where $D$ is the gradient operator in $\rnnn$. Utilizing \eqref{Fdef} and \eqref{hf}, we also have (see, e.g. \cite[p. 382]{J91})
\begin{equation*}\label{hgra}
D h(x)=\sum_{i}h_{i}e_{i}+hx, \quad F_{i}(x)=\sum_{j}(h_{ij}+h\delta_{ij})e_{j}.
\end{equation*}

Let $\sigma_{k}$ ($1\leq k \leq n$) be the $k$-th elementary symmetric function of principal radii of curvature, and let $\lambda=(\lambda_{1},\ldots, \lambda_{n})$ be the eigenvalues of matrix $\{h_{ij}+h\delta_{ij}\}$, which are the principal radii of curvature at the point $X(x)\in \partial K$. Then $\sigma_{1}=\Delta h+nh$, where $\Delta$ is the spherical Laplace operator, and the Gauss curvature of $\partial K$, $\kappa$, is expressed as
\begin{equation*}
\kappa=\frac{1}{\sigma_{n}}=\frac{1}{\det(h_{ij}+h\delta_{ij})}.
\end{equation*}

\section{Uniqueness of solutions to the isotropic $L_{p}$ Gaussian Minkowski problem}

The  lemma below gives the spectral formulation of the Aleksandrov-Fenchel inequality. This spectral formulation traces its roots to Hilbert's work (see, e.g. \cite{AN97, AB20}) and has been subsequently investigated in \cite{IM23, KM22, ML22}.

\begin{lem}\label{SF}\cite{AN97, AB20}
Let $f\in C^{2}(\sn)$ with $\int_{\sn}fh\sigma_{k}d\sigma=0$. Then we have
\begin{equation*}
k\int_{\sn}f^{2}h\sigma_{k}d\sigma\leq \int_{\sn}\sum_{i,j}h^{2}\sigma^{ij}_{k}\nabla_{i}f\nabla_{j}fd\sigma,
\end{equation*}
where $\sigma^{ij}_{k}=\frac{\partial \sigma_{k}}{\partial b_{ij}}$ with $b_{ij}:=h_{ij}+h\delta_{ij}$. Equality holds if and only if for some vector $v\in \rnnn$, we have
\[
f(x)=\langle \frac{x}{h(x)},v\rangle,  \quad \forall x\in \sn.
\]
\end{lem}
Utilizing Lemma \ref{SF} with $k=n$, we have the following result, which is the main ingredient of proving Theorem \ref{MTH}.
\begin{lem}\label{IN}
Let $X=Dh:\sn\rightarrow \partial K$ and $\alpha\in \rn$. Then we have
\begin{equation}
\begin{split}
\label{aa7}
&n\int_{\sn}e^{-|X|^{\alpha}}|X|^{2}dV_{n}-n\frac{\Big|\int_{\sn}e^{-\frac{|X|^{\alpha}}{2}}X dV_{n}\Big|^{2}}{\int_{\sn}dV_{n}}\\
&\leq \int_{\sn}e^{-|X|^{\alpha}}h(\Delta h+nh)dV_{n}+\frac{\alpha^{2}}{4}\int_{\sn}|X|^{2\alpha-1}e^{-|X|^{\alpha}}h\langle \nabla h, \nabla |X|\rangle dV_{n}\\
&\quad -\alpha\int_{\sn}|X|^{\alpha-1}e^{-|X|^{\alpha}}h\langle \nabla h, \nabla |X|\rangle dV_{n},
\end{split}
\end{equation}
where $dV_{n}=h\sigma_{n}d\sigma$.
\end{lem}

\begin{proof}
Let $\{E_{l}\}^{n+1}_{l=1}$ be an orthonormal basis of $\rnnn$. Suppose $\{e_{i}\}^{n}_{i=1}$ is a local orthonormal frame of $\sn$ such that $(h_{ij}+h\delta_{ij})(x_{0})=\lambda_{i}(x_{0})\delta_{ij}$. Motivated by Ivaki-Milman \cite[Lemma 3.2]{IM23}, for $l=1,\ldots,n+1$, we set
the functional $f_{l}:\sn\rightarrow \rn$ as
\begin{equation*}
f_{l}(x)=e^{-\frac{|X|^{\alpha}}{2}}\langle X(x),E_{l}\rangle-\frac{\int_{\sn}e^{-\frac{|X|^{\alpha}}{2}}\langle X(x),E_{l}\rangle dV_{n}}{\int_{\sn}dV_{n}}.
\end{equation*}
Due to $\int_{\sn}f_{l}dV_{n}=0$ for $1\leq l \leq n+1$, by using Lemma \ref{SF} to $f_{l}$ and summing over $l$, we derive
\begin{equation}\label{xz}
n\sum_{l}\int_{\sn}f^{2}_{l}dV_{n}=n\left( \int_{\sn}e^{-|X|^{\alpha}}|X|^{2}dV_{n}-\frac{\Big|\int_{\sn}e^{-\frac{|X|^{\alpha}}{2}}X dV_{n}\Big|^{2}}{\int_{\sn}dV_{n}} \right)\leq \sum_{l,i,j}\int_{\sn}h^{2}\sigma^{ij}_{n}\nabla_{i}f_{l}\nabla_{j}f_{l}d\sigma.
\end{equation}
Since $\nabla_{i}X=\sum_{j}(h_{ij}+h\delta_{ij})e_{j}=\lambda_{i}e_{i}$ at $x_{0}$. Then $\langle e_{i}, X\rangle=h_{i}$ and $\sum_{i}\lambda_{i}\langle e_{i},X\rangle^{2}=|X|\langle \nabla h, \nabla |X|\rangle$ at $x_{0}$. Using $\sum_{i}\frac{\partial \sigma_{n}}{\partial \lambda_{i}}\lambda^{2}_{i}=\sigma_{1}\sigma_{n}$ and $\frac{\partial \sigma_{n}}{\partial \lambda_{i}}\lambda_{i}=\sigma_{n}$ for $\forall i$, we get
\begin{equation}
\begin{split}
\label{xs}
\sum_{l,i,j}\sigma^{ij}_{n}\nabla_{i}f_{l}\nabla_{j}f_{l}&=\sum_{l,i}\frac{\partial \sigma_{n}}{\partial \lambda_{i}}\left((\nabla_{i}(e^{-\frac{|X|^{\alpha}}{2}}))\langle X, E_{l}\rangle+e^{-\frac{|X|^{\alpha}}{2}}\langle \nabla_{i}X,E_{l}\rangle\right)^{2}\\
&=\sum_{l,i}\frac{\partial \sigma_{n}}{\partial \lambda_{i}}\left(-e^{-\frac{|X|^{\alpha}}{2}}\frac{1}{2}(\nabla_{i}|X|^{\alpha})\langle X, E_{l}\rangle+e^{-\frac{|X|^{\alpha}}{2}}\langle \nabla_{i}X,E_{l}\rangle\right)^{2}\\
&=\sum_{l,i}\frac{\partial \sigma_{n}}{\partial \lambda_{i}}\left(-\frac{1}{2}\alpha |X|^{\alpha-2}e^{-\frac{|X|^{\alpha}}{2}}\langle \lambda_{i}e_{i},X\rangle\langle X, E_{l}\rangle+e^{-\frac{|X|^{\alpha}}{2}}\langle \lambda_{i}e_{i},E_{l}\rangle\right)^{2}\\
&=\sum_{i}\frac{\partial \sigma_{n}}{\partial \lambda_{i}}\left(\frac{\alpha^{2}}{4}e^{-|X|^{\alpha}}|X|^{2\alpha-2}\lambda^{2}_{i}\langle X,e_{i}\rangle^{2}+e^{-|X|^{\alpha}}\lambda^{2}_{i}-\alpha|X|^{\alpha-2}\lambda^{2}_{i}e^{-|X|^{\alpha}}\langle X,e_{i}\rangle^{2}\right)\\
&=e^{-|X|^{\alpha}}\sigma_{1}\sigma_{n}+\frac{\alpha^{2}}{4}\sigma_{n}e^{-|X|^{\alpha}}|X|^{2\alpha-1}\langle \nabla h, \nabla |X|\rangle-\alpha|X|^{\alpha-1}\sigma_{n}e^{-|X|^{\alpha}}\langle\nabla h, \nabla |X|\rangle.
\end{split}
\end{equation}
Thus, applying \eqref{xs} into \eqref{xz}, and employing $dV_{n}=h\sigma_{n}d\sigma$,  we obtain
\begin{equation*}
\begin{split}
&n\int_{\sn}e^{-|X|^{\alpha}}|X|^{2}dV_{n}-n\frac{\Big|\int_{\sn}e^{-\frac{|X|^{\alpha}}{2}}X dV_{n}\Big|^{2}}{\int_{\sn}dV_{n}}\\
&\leq\int_{\sn}h^{2}e^{-|X|^{\alpha}}\sigma_{1}\sigma_{n}d\sigma+\frac{\alpha^{2}}{4}\int_{\sn}|X|^{2\alpha-1}e^{-|X|^{\alpha}}h^{2}\sigma_{n}\langle \nabla h, \nabla |X|\rangle d\sigma\\
&\quad -\alpha \int_{\sn}|X|^{\alpha-1}e^{-|X|^{\alpha}}h^{2}\sigma_{n}\langle \nabla h, \nabla |X|\rangle d\sigma\\
&=\int_{\sn}e^{-|X|^{\alpha}}h(\Delta h+nh)dV_{n}+\frac{\alpha^{2}}{4}\int_{\sn}|X|^{2\alpha-1}e^{-|X|^{\alpha}}h\langle \nabla h, \nabla |X|\rangle dV_{n}\\
&\quad -\alpha\int_{\sn}|X|^{\alpha-1}e^{-|X|^{\alpha}}h\langle \nabla h, \nabla |X|\rangle dV_{n}.
\end{split}
\end{equation*}
The proof is complete.
\end{proof}
Using Lemma \ref{IN}, we can get the following equivalent form.

\begin{lem}
Let $X=Dh:\sn\rightarrow \partial K$ and $\alpha\in \rn$. Then we have
\begin{equation}
\begin{split}
\label{iu}
&\int_{\sn}\langle hX,  \nabla \log \frac{h^{n+2}}{\kappa}\rangle e^{-|X|^{\alpha}}dV_{n}
\leq n\frac{\Big|\int_{\sn}e^{-\frac{|X|^{\alpha}}{2}}X dV_{n}\Big|^{2}}{\int_{\sn}dV_{n}} + \frac{\alpha^{2}}{4}\int_{\sn}|X|^{2\alpha-1}e^{-|X|^{\alpha}}h\langle \nabla h, \nabla |X|\rangle dV_{n}.
\end{split}
\end{equation}
\end{lem}
\begin{proof}
By \eqref{aa7}, we have
\begin{equation}
\begin{split}
\label{aa3}
&\int_{\sn}e^{-|X|^{\alpha}}(n|\nabla h|^{2}-h\triangle h)dV_{n}-n\frac{\Big|\int_{\sn}e^{-\frac{|X|^{\alpha}}{2}}X dV_{n}\Big|^{2}}{\int_{\sn}dV_{n}}\\
&\leq \frac{\alpha^{2}}{4}\int_{\sn}|X|^{2\alpha-1}e^{-|X|^{\alpha}}h\langle \nabla h, \nabla |X|\rangle dV_{n}-\alpha\int_{\sn}|X|^{\alpha-1}e^{-|X|^{\alpha}}h\langle \nabla h, \nabla |X|\rangle dV_{n}.
\end{split}
\end{equation}
On the other hand, employing  $dV_{n}=h\sigma_{n}d\sigma$. By integration by parts, we obtain
\begin{equation}
\begin{split}
\label{aa4}
&\int_{\sn}\langle hX,  \nabla \log \frac{h^{n+2}}{\kappa}\rangle e^{-|X|^{\alpha}}dV_{n}\\
&=(n+2)\int_{\sn}|\nabla h|^{2}e^{-|X|^{\alpha}}dV_{n}+\int_{\sn}\langle h\nabla h, \sigma^{-1}_{n}\nabla \sigma_{n}\rangle e^{-|X|^{\alpha}}h\sigma_{n}d\sigma\\
&=(n+2)\int_{\sn}|\nabla h|^{2}e^{-|X|^{\alpha}}dV_{n}-\int_{\sn}{\rm div}(h\nabla h e^{-|X|^{\alpha}}h)\sigma_{n}d\sigma\\
&=(n+2)\int_{\sn}|\nabla h|^{2}e^{-|X|^{\alpha}}dV_{n}-2\int_{\sn}|\nabla h|^{2}e^{-|X|^{\alpha}}dV_{n}-\int_{\sn}h\triangle he^{-|X|^{\alpha}}dV_{n}\\
&\quad +\alpha\int_{\sn}|X|^{\alpha-1}e^{-|X|^{\alpha}}h^{2}\sigma_{n}\langle \nabla h, \nabla |X|\rangle d\sigma\\
&= n\int_{\sn}|\nabla h|^{2}e^{-|X|^{\alpha}}dV_{n}-\int_{\sn}h\triangle he^{-|X|^{\alpha}}dV_{n}+\alpha\int_{\sn}|X|^{\alpha-1}e^{-|X|^{\alpha}}h\langle \nabla h, \nabla |X|\rangle dV_{n}.
\end{split}
\end{equation}
Applying \eqref{aa4} into \eqref{aa3}, hence we get
\begin{equation*}
\begin{split}
\label{Up3}
&\int_{\sn}\langle hX,  \nabla \log \frac{h^{n+2}}{\kappa}\rangle e^{-|X|^{\alpha}}dV_{n}
\leq n\frac{\Big|\int_{\sn}e^{-\frac{|X|^{\alpha}}{2}}X dV_{n}\Big|^{2}}{\int_{\sn}dV_{n}} + \frac{\alpha^{2}}{4}\int_{\sn}|X|^{2\alpha-1}e^{-|X|^{\alpha}}h\langle \nabla h, \nabla |X|\rangle dV_{n}.
\end{split}
\end{equation*}
\end{proof}
Utilizing above results, we are in a position to prove Theorem \ref{MTH}.

\begin{proof}[Proof of Theorem \ref{MTH}.]
 Let $\alpha=2$ and recall $dV_{n}=h\sigma_{n}d\sigma=ch^{p}e^{\frac{|X|^{\alpha}}{2}}d\sigma$. One hand, by a direct computation, we get
\begin{equation*}
\begin{split}
\int_{\sn}\langle hX,  \nabla \log \frac{h^{n+2}}{\kappa}\rangle e^{-|X|^{\alpha}}dV_{n}&=\int_{\sn}\langle hX\frac{\kappa}{h^{n+2}}, \nabla \frac{h^{n+2}}{\kappa} \rangle e^{-|X|^{\alpha}}dV_{n}\\
&=\int_{\sn}\langle h^{2}X\frac{1}{h^{n+2}}, \nabla \frac{h^{n+2}}{\kappa} \rangle e^{-|X|^{\alpha}}d\sigma\\
&=\int_{\sn}\langle h^{-n}X, \nabla \frac{h^{n+2}}{\kappa}\rangle e^{-|X|^{\alpha}}d\sigma.
\end{split}
\end{equation*}
Since $h^{1-p}e^{-\frac{|X|^{\alpha}}{2}}\frac{1}{\kappa}=c$, $\frac{h^{n+2}}{\kappa}=ce^{\frac{|X|^{\alpha}}{2}}h^{n+1+p}$. It follows that
\begin{equation}
\begin{split}
\label{aa5}
&\int_{\sn}\langle h^{-n}X, \nabla \frac{h^{n+2}}{\kappa}\rangle e^{-|X|^{\alpha}}d\sigma=c\int_{\sn}  h^{-n}\langle X,\nabla (e^{\frac{|X|^{\alpha}}{2}}h^{n+1+p})\rangle e^{-|X|^{\alpha}}d\sigma\\
&=c(n+1+p)\int_{\sn}h^{p}|\nabla h|^{2}e^{-\frac{|X|^{\alpha}}{2}}d\sigma+\frac{\alpha}{2} c\int_{\sn}h^{1+p}\langle \nabla h, \nabla |X|\rangle |X|^{\alpha-1}e^{-\frac{|X|^{\alpha}}{2}}d\sigma\\
&=(n+1+p)\int_{\sn}|\nabla h|^{2}e^{-|X|^{\alpha}}dV_{n}+\frac{\alpha}{2} \int_{\sn}h\langle \nabla h, \nabla |X|\rangle |X|^{\alpha-1}e^{-|X|^{\alpha}}dV_{n}.
\end{split}
\end{equation}
Substituting \eqref{aa5} into \eqref{iu}, there is
\begin{equation}
\begin{split}
\label{m2}
(n+1+p)\int_{\sn}|\nabla h|^{2}e^{-|X|^{\alpha}}dV_{n}&\leq n\frac{\Big|\int_{\sn}e^{\frac{-|X|^{\alpha}}{2}}X dV_{n}\Big|^{2}}{\int_{\sn}dV_{n}}+\frac{\alpha^{2}}{4}\int_{\sn}|X|^{2\alpha-1}e^{-|X|^{\alpha}}h\langle \nabla h, \nabla |X|\rangle dV_{n}\\
&\quad-\frac{\alpha}{2}\int_{\sn}h\langle \nabla h, \nabla |X|\rangle |X|^{\alpha-1}e^{-|X|^{\alpha}}dV_{n}.
\end{split}
\end{equation}
On the other hand,
\begin{equation}
\begin{split}
\label{Up3}
&\int_{\sn}e^{-\frac{|X|^{\alpha}}{2}}X dV_{n}=c\int_{\sn}h^{p}Xd\sigma=c\int_{\sn}h^{p}(\nabla h+ h x)d\sigma(x)\\
&=c\frac{p+1+n}{n}\int_{\sn}h^{p}\nabla h d\sigma=\frac{p+1+n}{n}\int_{\sn}e^{-\frac{|X|^{\alpha}}{2}}\nabla hdV_{n}.
\end{split}
\end{equation}
By Cauchy-Schwarz inequality, there is
\begin{equation}
\begin{split}
\label{m1}
n \frac{\Big|\int_{\sn}e^{-\frac{|X|^{\alpha}}{2}}X dV_{n}\Big|^{2}}{\int_{\sn}dV_{n}}&\leq n\frac{(p+1+n)^{2}}{n^{2}}\frac{(\int_{\sn}e^{-\frac{|X|^{\alpha}}{2}}|\nabla h|dV_{n})^{2}}{\int_{\sn}dV_{n}}\\
&\leq \frac{(p+1+n)^{2}}{n}\int_{\sn}e^{-|X|^{\alpha}}|\nabla h|^{2}dV_{n}.
\end{split}
\end{equation}
Applying \eqref{m1} into \eqref{m2}, we have
\begin{equation}
\begin{split}
\label{yu}
&\frac{(n+1+p)(-1-p)}{n}\int_{\sn}e^{-|X|^{\alpha}}|\nabla h|^{2}dV_{n}\\
&\leq \frac{\alpha^{2}}{4}\int_{\sn}|X|^{2\alpha-1}e^{-|X|^{\alpha}}h\langle \nabla h, \nabla |X|\rangle dV_{n}-\frac{\alpha}{2}\int_{\sn}h\langle \nabla h, \nabla |X|\rangle |X|^{\alpha-1}e^{-|X|^{\alpha}}dV_{n}.
\end{split}
\end{equation}
 Note that $\alpha=2$ and $|X|\langle \nabla h, \nabla |X|\rangle=\sum_{i}\lambda_{i}h^{2}_{i}\geq c_{0} |\nabla h|^{2}$, where $c_{0}>0$ depends on $\partial K$. Due to $\max|X|\leq 1$, by \eqref{yu}, we obtain
\begin{equation*}
\begin{split}
&\frac{(n+1+p)(-1-p)}{n}\int_{\sn}e^{-|X|^{\alpha}}|\nabla h|^{2}dV_{n}\leq 0.
\end{split}
\end{equation*}
 If $-(n+1)< p <-1$, then $|\nabla h|\equiv 0$. Thus $\partial K$ is a sphere. By studying the monotonic properties of function $g(t)=t^{n+1-p}e^{-\frac{t^{2}}{2}}$, we see the number of constant solutions to \eqref{MP}. It is clear to know that $g(t)$ is strictly increasing in $(0,\sqrt{n+1-p}]$, and is strictly decreasing in $(\sqrt{n+1-p},\infty)$, where $\sqrt{n+1-p}>1$ for $-(n+1)< p <-1$ and $n\geq 1$, moreover $\max_{t\in (0, \infty)}g(t)=(n+1-p)^{\frac{n+1-p}{2}}e^{-\frac{n+1-p}{2}}> e^{-\frac{1}{2}}$, hence the proof is completed.
\end{proof}

\appendix
\section{A local version of the Ehrhard inequality}
Similar to  Lemma \ref{SF}, it is an independent interest to see the spectral formulation of the Ehrhard inequality for the Gaussian measure, which can be regarded as a local version of the Ehrhard inequality.

Set $I_{\gamma}:[0,1]\rightarrow \rt$ by $I_{\gamma}:=\varphi \circ \Phi^{-1}$, where $\varphi(t)=\frac{1}{\sqrt{2\pi}}e^{-t^{2}/2}$, $\Phi(t)=\int^{t}_{-\infty}\varphi(s)ds$ and $\Phi^{-1}$ is the inverse function of $\Phi$. By using the Ehrhard inequality on Gaussian volumes, Kolesnikov-Milman \cite{KM2} obtained a class of sharp Poincar\'{e}-type inequalities for the Gaussian measure on the boundary of convex sets as below.

\begin{lem}\cite[Theorem 1.1]{KM2}\label{KM}
Let $g\in C^{2}(\partial K)$. Then we have
\begin{equation*}
\begin{split}
&\frac{1}{(\sqrt{2\pi})^{n+1}}\int_{\partial K}H g^{2}e^{-\frac{|X|^{2}}{2}}d\mathcal{H}^{n}-\frac{1}{(\sqrt{2\pi})^{n+1}}\int_{\partial K}\langle X, \nu_{K}(X)\rangle g^{2}e^{-\frac{|X|^{2}}{2}}d\mathcal{H}^{n}\\
&\quad -(\log I_{\gamma})^{'}(\gamma(K))\left(\frac{1}{(\sqrt{2\pi})^{n+1}}\int_{\partial K}g e^{-\frac{|X|^{2}}{2}}d\mathcal{H}^{n}\right)^{2} \\
&\quad \leq \frac{1}{(\sqrt{2\pi})^{n+1}}\int_{\partial K}\langle \Pi^{-1}_{\partial K}\nabla_{\partial K}g,\nabla_{\partial K} g\rangle e^{-\frac{|X|^{2}}{2}} d\mathcal{H}^{n},
\end{split}
\end{equation*}
where $\nabla_{\partial K}$ denotes the induced Levi-Civita connection on the boundary $\partial K$, $H$ is the mean curvature of $\partial K$ at $X$, $\Pi_{\partial K}^{-1}$ is the inverse of the second fundamental form $\Pi_{\partial K}$ of $\partial K$ at $X$ and $(\log I_{\gamma})^{'}(v)=-\Phi^{-1}(v)/I_{\gamma}(v)$.
\end{lem}

Based on Lemma \ref{KM}, we can get the related spectral formulation of  the Ehrhard inequality for the Gaussian measure.
\begin{lem}\label{SPG}
Let $f\in C^{2}(\sn)$. Then we have
\begin{equation}
\begin{split}
\label{AQM}
n\int_{\sn}f^{2}h\sigma_{n}e^{-\frac{|Dh|^{2}}{2}}d\sigma&\leq \int_{\sn}\sum_{i,j}h^{2}\sigma^{ij}_{n}f_{i}f_{j}e^{-\frac{|Dh|^{2}}{2}}d\sigma +\int_{\sn}f^{2}h^{2}e^{-\frac{|Dh|^{2}}{2}}h\sigma_{n}d\sigma\\
&\quad+ \int_{\sn}\sum_{i,j}\sigma^{ij}_{n}h_{i}|Dh|_{j}|Dh|f^{2}he^{-\frac{|Dh|^{2}}{2}}d\sigma\\
&\quad + (\log I_{\gamma})^{'}(\gamma(K))\frac{1}{(\sqrt{2\pi})^{n+1}}\left( \int_{\sn}fh\sigma_{n}e^{-\frac{|Dh|^{2}}{2}}d\sigma\right)^{2}.
\end{split}
\end{equation}
\end{lem}
\begin{proof}
By Lemma \ref{KM} and the assumption, we have
\begin{equation*}
\begin{split}
\label{Up3}
&\int_{\sn}\sum_{i}\sigma^{ii}_{n}(fh)^{2}e^{-\frac{|Dh|^{2}}{2}}d\sigma-\int_{\sn}h(fh)^{2}\sigma_{n}e^{-\frac{|Dh|^{2}}{2}}d\sigma\\
\leq &\int_{\sn}\sum_{i,j}\sigma^{ij}_{n}(fh)_{i}(fh)_{j}e^{-\frac{|Dh|^{2}}{2}}d\sigma+(\log I_{\gamma})^{'}(\gamma(K))\frac{1}{(\sqrt{2\pi})^{n+1}}\left( \int_{\sn}fh\sigma_{n}e^{-\frac{|Dh|^{2}}{2}}d\sigma \right)^{2}.
\end{split}
\end{equation*}
It follows that
\begin{equation}
\begin{split}
\label{AS}
&\int_{\sn}\sum_{i}\sigma^{ii}_{n}(fh)^{2}e^{-\frac{|Dh|^{2}}{2}}d\sigma-\int_{\sn}h(fh)^{2}\sigma_{n}e^{-\frac{|Dh|^{2}}{2}}d\sigma
\\
&\leq \int_{\sn}\sum_{i,j}\sigma^{ij}_{n}h^{2}f_{i}f_{j}e^{-\frac{|Dh|^{2}}{2}}d\sigma+2\int_{\sn}\sum_{i,j}\sigma^{ij}_{n}fhh_{i}f_{j}e^{-\frac{|Dh|^{2}}{2}}d\sigma+\int_{\sn}\sum_{i,j}\sigma^{ij}_{n}h_{i}h_{j}f^{2}e^{-\frac{|Dh|^{2}}{2}}d\sigma\\
&\quad + (\log I_{\gamma})^{'}(\gamma(K))\frac{1}{(\sqrt{2\pi})^{n+1}}\left( \int_{\sn}fh\sigma_{n}e^{-\frac{|Dh|^{2}}{2}}d\sigma\right)^{2}.
\end{split}
\end{equation}
One hand, by integration by parts, we obtain
\begin{equation}
\begin{split}
\label{AQ}
&2\int_{\sn}\sum_{i,j}\sigma^{ij}_{n}fhh_{i}f_{j}e^{-\frac{|Dh|^{2}}{2}}d\sigma\\
&=-2\int_{\sn}\sum_{i,j}\sigma^{ij}_{n}h_{ij}f^{2}he^{-\frac{|Dh|^{2}}{2}}d\sigma-2\int_{\sn}\sum_{i,j}\sigma^{ij}_{n}fhh_{i}f_{j}e^{-\frac{|Dh|^{2}}{2}}d\sigma\\
&\quad +2\int_{\sn}\sum_{i,j}\sigma^{ij}_{n}h_{i}f^{2}h|Dh|e^{-\frac{|Dh|^{2}}{2}}|Dh|_{j}d\sigma-2\int_{\sn}\sum_{i,j}\sigma^{ij}_{n}h_{i}h_{j}f^{2}e^{-\frac{|Dh|^{2}}{2}}d\sigma.
\end{split}
\end{equation}
Substituting \eqref{AQ} into \eqref{AS}, we have
\begin{equation}
\begin{split}
\label{AS2}
&\int_{\sn}\sum_{i}\sigma^{ii}_{n}(fh)^{2}e^{-\frac{|Dh|^{2}}{2}}d\sigma-\int_{\sn}h(fh)^{2}\sigma_{n}e^{-\frac{|Dh|^{2}}{2}}d\sigma
\\
&\leq \int_{\sn}\sum_{i,j}\sigma^{ij}_{n}h^{2}f_{i}f_{j}e^{-\frac{|Dh|^{2}}{2}}d\sigma-2\int_{\sn}\sum_{i,j}\sigma^{ij}_{n}h_{ij}f^{2}he^{-\frac{|Dh|^{2}}{2}}d\sigma-2\int_{\sn}\sum_{i,j}\sigma^{ij}_{n}fhh_{i}f_{j}e^{-\frac{|Dh|^{2}}{2}}d\sigma\\
&\quad +2\int_{\sn}\sum_{i,j}\sigma^{ij}_{n}h_{i}f^{2}h|Dh|e^{-\frac{|Dh|^{2}}{2}}|Dh|_{j}d\sigma-\int_{\sn}\sum_{i,j}\sigma^{ij}_{n}h_{i}h_{j}f^{2}e^{-\frac{|Dh|^{2}}{2}}d\sigma\\
&\quad + (\log I_{\gamma})^{'}(\gamma(K))\frac{1}{(\sqrt{2\pi})^{n+1}}\left( \int_{\sn}fh\sigma_{n}e^{-\frac{|Dh|^{2}}{2}}d\sigma\right)^{2}.
\end{split}
\end{equation}
On the other hand, using again integration by parts, we obtain
\begin{equation}
\begin{split}
\label{aa}
&\int_{\sn}\sum_{i,j}\sigma^{ij}_{n}h_{i}h_{j}f^{2}e^{-\frac{|Dh|^{2}}{2}}d\sigma\\
&=-\int_{\sn}\sum_{i,j}\sigma^{ij}_{n}h_{ij}hf^{2}e^{-\frac{|Dh|^{2}}{2}}d\sigma-2\int_{\sn}\sum_{i,j}\sigma^{ij}_{n}fhh_{i}f_{j}e^{-\frac{|Dh|^{2}}{2}}d\sigma\\
&\quad +\int_{\sn}\sum_{i,j}\sigma^{ij}_{n}h_{i}f^{2}h|Dh|e^{-\frac{|Dh|^{2}}{2}}|Dh|_{j}d\sigma.
\end{split}
\end{equation}
Now applying \eqref{aa} into \eqref{AS2}, we get
\begin{equation}
\begin{split}
\label{AS3}
&\int_{\sn}\sum_{i}\sigma^{ii}_{n}(fh)^{2}e^{-\frac{|Dh|^{2}}{2}}d\sigma-\int_{\sn}h(fh)^{2}\sigma_{n}e^{-\frac{|Dh|^{2}}{2}}d\sigma
\\
&\leq \int_{\sn}\sum_{i,j}\sigma^{ij}_{n}h^{2}f_{i}f_{j}e^{-\frac{|Dh|^{2}}{2}}d\sigma-\int_{\sn}\sum_{i,j}\sigma^{ij}_{n}h_{ij}f^{2}he^{-\frac{|Dh|^{2}}{2}}d\sigma +\int_{\sn}\sum_{i,j}\sigma^{ij}_{n}h_{i}f^{2}h|Dh|e^{-\frac{|Dh|^{2}}{2}}|Dh|_{j}d\sigma\\
&\quad + (\log I_{\gamma})^{'}(\gamma(K))\frac{1}{(\sqrt{2\pi})^{n+1}}\left( \int_{\sn}fh\sigma_{n}e^{-\frac{|Dh|^{2}}{2}}d\sigma\right)^{2}.
\end{split}
\end{equation}
Thus \eqref{AS3} becomes
\begin{equation*}
\begin{split}
\label{AS4}
&\int_{\sn}\sum_{i,j}\sigma^{ij}_{n}b_{ij}f^{2}he^{-\frac{|Dh|^{2}}{2}}d\sigma-\int_{\sn}h(fh)^{2}\sigma_{n}e^{-\frac{|Dh|^{2}}{2}}d\sigma
\\
&\leq \int_{\sn}\sum_{i,j}\sigma^{ij}_{n}h^{2}f_{i}f_{j}e^{-\frac{|Dh|^{2}}{2}}d\sigma +\int_{\sn}\sum_{i,j}\sigma^{ij}_{n}h_{i}f^{2}h|Dh|e^{-\frac{|Dh|^{2}}{2}}|Dh|_{j}d\sigma\\
&\quad + (\log I_{\gamma})^{'}(\gamma(K))\frac{1}{(\sqrt{2\pi})^{n+1}}\left( \int_{\sn}fh\sigma_{n}e^{-\frac{|Dh|^{2}}{2}}d\sigma\right)^{2}.
\end{split}
\end{equation*}
Hence
\begin{equation*}
\begin{split}
\label{AS4}
n\int_{\sn}f^{2}h\sigma_{n}e^{-\frac{|Dh|^{2}}{2}}d\sigma&\leq \int_{\sn}\sum_{i,j}\sigma^{ij}_{n}h^{2}f_{i}f_{j}e^{-\frac{|Dh|^{2}}{2}}d\sigma +\int_{\sn}f^{2}h^{2}e^{-\frac{|Dh|^{2}}{2}}h\sigma_{n}d\sigma\\
&\quad+ \int_{\sn}\sum_{i,j}\sigma^{ij}_{n}h_{i}f^{2}h|Dh|e^{-\frac{|Dh|^{2}}{2}}|Dh|_{j}d\sigma\\
&\quad + (\log I_{\gamma})^{'}(\gamma(K))\frac{1}{(\sqrt{2\pi})^{n+1}}\left( \int_{\sn}fh\sigma_{n}e^{-\frac{|Dh|^{2}}{2}}d\sigma\right)^{2}.
\end{split}
\end{equation*}
\end{proof}

Notice that if  $\gamma(K)\geq 1/2$, it follows that $(\log I_{\gamma})^{'}(\gamma(K))\leq 0$, then we have the following corollary.
\begin{coro}
Let $f\in C^{2}(\sn)$. Suppose $\gamma(K)\geq 1/2$. Then there is
\begin{equation*}
\begin{split}
n\int_{\sn}f^{2}h\sigma_{n}e^{-\frac{|Dh|^{2}}{2}}d\sigma&\leq \int_{\sn}\sum_{i,j}h^{2}\sigma^{ij}_{n}f_{i}f_{j}e^{-\frac{|Dh|^{2}}{2}}d\sigma +\int_{\sn}f^{2}h^{2}e^{-\frac{|Dh|^{2}}{2}}h\sigma_{n}d\sigma\\
&\quad+ \int_{\sn}\sum_{i,j}\sigma^{ij}_{n}h_{i}|Dh|_{j}|Dh|f^{2}he^{-\frac{|Dh|^{2}}{2}}d\sigma.
\end{split}
\end{equation*}
\end{coro}

\begin{rem}
 From the proof of Theorem \ref{MTH}, one sees that the restriction  $R(K)\leq 1$  arises from  the application of the spectral form of the local Brunn-Minkowski inequality. Lemma \ref{SPG} might be helpful to extend Theorem \ref{MTH} to broader classes of convex bodies without imposing assumptions on $R(K)$, perhaps by finding a suitable test function $f$ into the spectral form of the local Ehrhard inequality \eqref{AQM}.
\end{rem}
\section*{Acknowledgment} The author would like to thank  Mohammad N. Ivaki for his very helpful discussions and comments on this work. 

{\bf Conflict of interest:} The author declares that there is no conflict of interest.

{\bf Data availability:} No data was used for the research described in the article.

\end{document}